\documentclass{article}%
\usepackage{amssymb}
\usepackage{graphicx}
\usepackage{amsmath}%
\usepackage{amsfonts}
\providecommand{\U}[1]{\protect\rule{.1in}{.1in}}
\providecommand{\U}[1]{\protect\rule{.1in}{.1in}}
\newtheorem{theorem}{Theorem}

\newtheorem{definition}[theorem]{Definition}

\newtheorem{lemma}[theorem]{Lemma}
\newtheorem{notation}[theorem]{Notation}

\newtheorem{proposition}[theorem]{Proposition}
\newtheorem{remark}[theorem]{Remark}

\newenvironment{proof}[1][Proof]{\textbf{#1.} }{\ \rule{0.5em}{0.5em}}
\begin{document}

\title{\textbf{Relative Microlinearity}\\- Towards the General Theory of Fiber Bundles \\in Infinite-Dimensionl Differential Geometry -}
\author{Hirokazu Nishimura\\Institute of Mathematics, University of Tsukuba\\Tsukuba, Ibaraki, 305-8571, Japan}
\maketitle

\begin{abstract}
In our previous papers [Far East Journal of Mathematical Sciences, \textbf{35}
(2009), 211-223] and [International Journal of Pure and Applied Mathematics,
\textbf{60} (2010), 15-24] \ we have developed the theory of Weil
prolongation, Weil exponentiability and microlinearity for Fr\"{o}licher
spaces. In this paper we will relativize it so as to obtain the theory of
fiber bundles for Fr\"{o}licher spaces. It is shown that any Weil functor
naturally gives rise to a fiber bundle. We will see that the category of fiber
bundles over a fixed Fr\"{o}licher space $M$ and their smooth mappings over
$M$ is cartesian closed. We will see also that the category of vector bundles
over $M$ and their smooth linear mappings over $M$ is cartesian closed. It is
also shown that the tangent bundle functor naturally yields a vector bundle.

\end{abstract}

\section{Introduction}

Smooth manifolds have been the central object of study in
\textit{finite-dimensional} differential geometry. In our previous paper
\cite{nishi2}, we proposed that \textit{microlinear} Fr\"{o}licher spaces will
be the central object of study in \textit{infinite-dimensional} (or rather
\textit{dimensionless}) differential geometry. The paper was followed by
\cite{nishi3} and \cite{nishi4}, which dealt with vector fields and
differential forms respectively. It is to be followed further by papers which
are to be concerned with connections, curvature, jet bundles, and so on. The
basic ideas are simple enough. What has been established in synthetic
differential geometry, which forces us to work within a well-adapted model so
as to make nilpotent infinitesimals visible, can be externalized by using Weil
functors. The simplest and best-known Weil functor is the tangent bundle
functor, which corresponds to the infinitesimal object of first-order
infinitesimals (i.e., real numbers whose squares vanish). Even an ordinary
teener of our modern times knows well that such real numbers must be
exclusively $0$, but, curiously enough, there appeared many such nilpotent
infinitesimals to Newton, Leibniz, Euler and other outstanding mathematicians
of the 17th and 18th centuries. It is in the 19th century, in the midst of the
Industrial Revolution, that nilpotent infinitesimals were ostracized as
anathema and replaced by seemingly rigorous $\varepsilon$-$\delta$ arguments,
which are, fortunately or unfortunately, normal nowadays.

The next most central object in differential geometry has been \textit{fiber
bundles}. The principal objective in this paper is to show how to deal with
fiber bundles without any reference to local trivializations at all, just as
we have replaced smooth manifolds by microlinear Fr\"{o}licher spaces.
Undoubtedly, fiber bundles should be more than a mere smooth mapping $\pi
_{M}^{E}:E\rightarrow M$\ of Fr\"{o}licher spaces. The missing concept is a
sort of relativization of microlinearity as well as Weil exponentiability.
Given a fixed Fr\"{o}licher space $M$, we will develop the theory of $M$-Weil
exponentiability and $M$-microlinearity on the lines of \cite{nishi1} and
\cite{nishi2}. It is to be shown that any Weil functor naturally gives rise to
a fiber bundle.We will finally show that the category of fiber bundles over
$M$ and their smooth mappings over $M$\ as well as the category of vector
bundles over $M$ and their smooth linear mappings over $M$ is cartesian closed.

\section{Preliminaries}

\subsection{Fr\"{o}licher Spaces}

Fr\"{o}licher and his followers have vigorously and consistently developed a
general theory of smooth spaces, often called \textit{Fr\"{o}licher spaces}
for his celebrity, which were intended to be the \textit{underlying set
theory} for infinite-dimensional differential geometry in a certain sense. A
Fr\"{o}licher space is usually depicted as an underlying set endowed with a
class of real-valued functions on it (simply called \textit{structure}
\textit{functions}) and a class of mappings from the set $\mathbb{R}$ of real
numbers to the underlying set (called \textit{structure} \textit{curves})
subject to the condition that structure curves and structure functions should
compose so as to yield smooth mappings from $\mathbb{R}$ to itself. It is
required that the class of structure functions and that of structure curves
should determine each other so that each of the two classes is maximal with
respect to the other as far as they abide by the above condition. For a
standard reference on Fr\"{o}licher spaces, the reader is referred to
\cite{fro}. What is most important among many nice properties about the
category $\mathbf{FS}$ of Fr\"{o}licher spaces and smooth mappings is that

\begin{theorem}
\label{t2.1.1}The category $\mathbf{FS}$ is cartesian closed.
\end{theorem}

\begin{notation}
We use $\cdot\times\cdot$ and $[\cdot,\cdot]$\ for product and exponentiation
in the category $\mathbf{FS}$.
\end{notation}

Recall that a Fr\"{o}licher space endowed with a compatible linear structure
is called a \textit{preconvenient vector space}. It is well known that

\begin{theorem}
\label{t2.1.2}The category $\mathbf{preCon}$\ of preconvenient vector spaces
and their smooth linear mappings is cartesian closed.
\end{theorem}

\begin{notation}
We use $\cdot\times\cdot$ and $[\cdot,\cdot]^{\mathrm{Lin}}$\ for product and
exponentiation in the category $\mathbf{preCon}$.
\end{notation}

These two results can easily be relativized.

\begin{theorem}
\label{t2.1.3}Let $M$\ be a Fr\"{o}licher space. The slice category
$\mathbf{FS}/M$ is cartesian closed.
\end{theorem}

\begin{notation}
We use $\cdot\times_{M}\cdot$ and $[\cdot,\cdot]_{M}$\ for product and
exponentiation in the category $\mathbf{FS}/M$.
\end{notation}

\begin{theorem}
\label{t2.1.4}The category $\mathbf{preVect}_{M}$\ consisting of vector spaces
in the category $\mathbf{FS}/M$\ and their linear morphisms in the category
$\mathbf{FS}/M$\ is cartesian closed.
\end{theorem}

\begin{notation}
We use $\cdot\times_{M}\cdot$ and $[\cdot,\cdot]_{M}^{\mathrm{Lin}}$\ for
product and exponentiation in the category $\mathbf{preVect}_{M}$.
\end{notation}

\subsection{Weil Algebras}

The notion of a \textit{Weil algebra} was introduced by Weil himself in
\cite{wei}. We denote by $\mathbf{W}$ the category of Weil algebras and their
$\mathbb{R}$-algebra homomorphisms preserving maximal ideals. It is well known
that the category $\mathbf{W}$\ is left exact. Roughly speaking, each Weil
algebra $W$\ corresponds to an infinitesimal object $\mathcal{D}_{W}$\ in the
shade. By way of example, the Weil algebra $\mathbb{R}[X]/(X^{2})$ (=the
quotient ring of the polynomial ring $\mathbb{R}[X]$\ of an indeterminate
$X$\ modulo the ideal $(X^{2})$\ generated by $X^{2}$) corresponds to the
infinitesimal object of first-order nilpotent infinitesimals, while the Weil
algebra $\mathbb{R}[X]/(X^{3})$ corresponds to the infinitesimal object of
second-order nilpotent infinitesimals. Although an infinitesimal object is
undoubtedly imaginary in the real world, as has harassed both mathematicians
and philosophers of the 17th and the 18th centuries because mathematicians at
that time preferred to talk infinitesimal objects as if they were real
entities, each Weil algebra yields its corresponding \textit{Weil functor} on
the category of smooth manifolds of some kind to itself, which is no doubt a
real entity. Intuitively speaking, the Weil functor corresponding to a Weil
algebra stands for the exponentiation by the infinitesimal object
corresponding to the Weil algebra at issue. For Weil functors on the category
of finite-dimensional smooth manifolds, the reader is referred to \S 35 of
\cite{kolar}, while the reader can find a readable treatment of Weil functors
on the category of smooth manifolds modelled on convenient vector spaces in
\S 31 of \cite{kri}.

\textit{Synthetic differential geometry }(usually abbreviated to SDG), which
is a kind of differential geometry with a cornucopia of nilpotent
infinitesimals, was forced to invent its models, in which nilpotent
infinitesimals were visible. For a standard textbook on SDG, the reader is
referred to \cite{lav}, while he or she is referred to \cite{kock} for the
model theory of SDG, which was vigorously constructed by Dubuc \cite{dub} and
others. Although we do not get involved in SDG herein, we will exploit
locutions in terms of infinitesimal objects so as to make the paper highly
readable. Thus we prefer to write $\mathcal{W}_{D}$\ and $\mathcal{W}_{D_{2}}
$\ in place of $\mathbb{R}[X]/(X^{2})$ and $\mathbb{R}[X]/(X^{3})$
respectively, where $D$ stands for the infinitesimal object of first-order
nilpotent infinitesimals, and $D_{2}$\ stands for the infinitesimal object of
second-order nilpotent infinitesimals. To Newton and Leibniz, $D$ stood for
\[
\{d\in\mathbb{R}\mid d^{2}=0\}
\]
while $D_{2}$\ stood for
\[
\{d\in\mathbb{R}\mid d^{3}=0\}
\]
We will write $\mathcal{W}_{d\in D_{2}\mapsto d^{2}\in D}$ for the homomorphim
of Weil algebras $\mathbb{R}[X]/(X^{2})\rightarrow\mathbb{R}[X]/(X^{3})$
induced by the homomorphism $X\rightarrow X^{2}$ of the polynomial ring
\ $\mathbb{R}[X]$ to itself. Such locutions are justifiable, because the
category $\mathbf{W}$ in the real world and the category of infinitesimal
objects in the shade are dual to each other in a sense, being interconnected
by the contravariant functors $\mathcal{W}$\ (from the category of
infinitesimal objects to the category of Weil algebras) and $\mathcal{D}$
(from the category of Weil algebras to the category of infinitesimal objects).
To familiarize himself or herself with such locutions, the reader is strongly
encouraged to read the first two chapters of \cite{lav}, even if he or she is
not interested in SDG at all.

\subsection{Microlinearity}

In \cite{nishi1} we have discussed how to assign, to each pair $(M,W)$\ of a
Fr\"{o}licher space $M$ and a Weil algebra $W$,\ another Fr\"{o}licher space
$M\otimes W$\ called the \textit{Weil prolongation of} $M$ \textit{with
respect to} $W$, which is naturally extended to a bifunctor $\mathbf{FS}%
\times\mathbf{W\rightarrow FS}$, and then to show that the functor
$\cdot\otimes W:\mathbf{FS\rightarrow FS}$ is product-preserving for any Weil
algebra $W$. Weil prolongations are well-known as \textit{Weil functors} for
finite-dimensional and infinite-dimensional smooth manifolds in orthodox
differential geometry, as we have already touched upon in the preceding
subsection. We note in passing that

\begin{lemma}
\label{t2.3.1}For any Fr\"{o}licher spaces $M$\ and $N$\ and for any Weil
algebra $W$, we have
\[
\left[  N,M\right]  \otimes W=\left[  N,M\otimes W\right]
\]

\end{lemma}

The central object of study in SDG is \textit{microlinear} spaces. Although
the notion of a manifold (=a pasting of copies of a certain linear space) is
defined on the local level, the notion of microlinearity is defined on the
genuinely infinitesimal level. For the historical account of microlinearity,
the reader is referred to \S \S 2.4 of \cite{lav} or Appendix D of
\cite{kock}. To get an adequately restricted cartesian closed subcategory of
Fr\"{o}licher spaces, we have emancipated microlinearity from within a
well-adapted model of SDG to Fr\"{o}licher spaces within the real world in
\cite{nishi2}. Recall that a Fr\"{o}licher space $M$ is called
\textit{microlinear} providing that any finite limit diagram $\mathbb{D}$ in
$\mathbf{W}$ yields a limit diagram $M\otimes\mathbb{D}$ in $\mathbf{FS}$,
where $M\otimes\mathbb{D}$ is obtained from $\mathbb{D}$ by putting $M\otimes$
to the left of every object in $\mathbb{D}$ and putting $\mathrm{id}%
_{M}\otimes$ to the left of every morphism in $\mathbb{D}$. As we have
discussed there, all convenient vector spaces are microlinear, so that all
$C^{\infty}$-manifolds in the sense of \cite{kri} (cf. Section 27) are also microlinear.

We have no reason to hold that all Fr\"{o}licher spaces credit Weil
prolongations as exponentiation by infinitesimal objects in the shade.
Therefore we need a notion which distinguishes Fr\"{o}licher spaces that do so
from those that do not. Here we slighly modify the notion of Weil
exponentiability introduced in \cite{nishi1}. A Fr\"{o}licher space $M$ is
called \textit{Weil exponentiable }if
\begin{equation}
M\otimes(W_{1}\otimes_{\infty}W_{2})=(M\otimes W_{1})\otimes W_{2}%
\label{2.3.1}%
\end{equation}
holds naturally for any Weil algebras $W_{1}$ and $W_{2}$.

\begin{proposition}
\label{t2.3.2}If a Fr\"{o}licher space $M$\ is Weil exponentiable, then so is
$\left[  N,M\right]  $\ for any Fr\"{o}licher space $N$.
\end{proposition}

\begin{proof}
For any Weil algebras $W_{1}$ and $W_{2}$, we have
\begin{align*}
&  \left[  N,M\right]  \otimes(W_{1}\otimes_{\infty}W_{2})\\
&  =\left[  N,M\otimes(W_{1}\otimes_{\infty}W_{2})\right] \\
&  \text{[by Lemma \ref{t2.3.1}]}\\
&  =\left[  N,(M\otimes W_{1})\otimes W_{2}\right] \\
&  =\left[  N,M\otimes W_{1}\right]  \otimes W_{2}\\
&  \text{[by Lemma \ref{t2.3.1}]}\\
&  =\left(  \left[  N,M\right]  \otimes W_{1}\right)  \otimes W_{2}%
\end{align*}

\end{proof}

\begin{proposition}
\label{t2.3.3}If a Fr\"{o}licher space $M$\ is Weil exponentiable, then so is
$M\otimes W$ for any Weil algebra $W$.
\end{proposition}

\begin{proof}
For any Weil algebras $W_{1}$ and $W_{2}$, we have
\begin{align*}
&  \left(  M\otimes W\right)  \otimes(W_{1}\otimes_{\infty}W_{2})\\
&  =M\otimes\left(  W\otimes_{\infty}(W_{1}\otimes_{\infty}W_{2})\right) \\
&  =M\otimes\left(  \left(  W\otimes_{\infty}W_{1}\right)  \otimes_{\infty
}W_{2}\right) \\
&  =\left(  M\otimes\left(  W\otimes_{\infty}W_{1}\right)  \right)  \otimes
W_{2}\\
&  =\left(  \left(  M\otimes W\right)  \otimes W_{1}\right)  \otimes W_{2}%
\end{align*}

\end{proof}

\begin{proposition}
\label{t2.3.4}If Fr\"{o}licher spaces $M$\ and $N$ are Weil exponentiable,
then so is $M\times N$.
\end{proposition}

\begin{proof}
For any Weil algebras $W_{1}$ and $W_{2}$, we have
\begin{align*}
&  \left(  M\times N\right)  \otimes(W_{1}\otimes_{\infty}W_{2})\\
&  =\left(  M\otimes(W_{1}\otimes_{\infty}W_{2})\right)  \times\left(
N\otimes(W_{1}\otimes_{\infty}W_{2})\right) \\
\text{\lbrack since the functor }\cdot\otimes(W_{1}\otimes_{\infty}W_{2})  &
:\mathbf{FS\rightarrow FS}\text{ is product-preserving]}\\
&  =\left(  (M\otimes W_{1})\otimes W_{2}\right)  \times\left(  (N\otimes
W_{1})\otimes W_{2}\right) \\
&  =\left(  (M\otimes W_{1})\times(N\otimes W_{1})\right)  \otimes W_{2}\\
\text{\lbrack since the functor }\cdot\otimes W_{2}  &  :\mathbf{FS\rightarrow
FS}\text{ is product-preserving]}\\
&  =\left(  \left(  M\times N\right)  \otimes W_{1}\right)  \otimes W_{2}%
\end{align*}

\end{proof}

Our present notion of Weil exponentiability is essentially the same as the one
in \cite{nishi1}, as the following proposition shows.

\begin{proposition}
\label{t2.3.5}A Fr\"{o}licher space $M$\ is Weil exponentiable iff
\begin{equation}
\left[  N,M\otimes(W_{1}\otimes_{\infty}W_{2})\right]  =\left[  N,M\otimes
W_{1}\right]  \otimes W_{2}\label{2.3.2}%
\end{equation}
holds naturally for any Fr\"{o}licher space $N$ and any Weil algebras $W_{1}$
and $W_{2}$.
\end{proposition}

\begin{proof}
By taking $N=1$ in (\ref{2.3.2}), we can see that (\ref{2.3.2}) implies
(\ref{2.3.1}). To see the converse, it suffices to note that
\begin{align*}
&  \left[  N,M\otimes(W_{1}\otimes_{\infty}W_{2})\right] \\
&  =\left[  N,M\right]  \otimes(W_{1}\otimes_{\infty}W_{2})\\
&  \text{[by Lemma \ref{t2.3.1}]}\\
&  =\left(  \left[  N,M\right]  \otimes W_{1}\right)  \otimes W_{2}\\
&  \text{[by Proposition \ref{t2.3.2}]}\\
&  =\left[  N,M\otimes W_{1}\right]  \otimes W_{2}\\
&  \text{[by Lemma \ref{t2.3.1}]}%
\end{align*}

\end{proof}

\begin{theorem}
\label{t2.3.6}The category $\mathbf{FS}_{\mathrm{WE}}$ of Weil exponentiable
Fr\"{o}licher spaces and their smooth mappings ($\mathbf{FS}_{\mathrm{ML}}$ of
microlinear Fr\"{o}licher spaces and their smooth mappings, $\mathbf{FS}%
_{\mathrm{WE,ML}}$ of Weil exponentiable and microlinear Fr\"{o}licher spaces
and their smooth mappings, respectively) is cartesian closed.
\end{theorem}

\begin{proof}
The case of $\mathbf{FS}_{\mathrm{WE}}$\ follows from Theorem \ref{t2.1.1} and
Propositions \ref{t2.3.2}\ and \ref{t2.3.4}. The remaining two cases can be
dealt with similarly.
\end{proof}

\section{Relativized Weil Prolongation}

Now we would like to relativize the notion of Weil prolongation. Since our
discussion is parallel to Section 3 of \cite{nishi1}, we can be brief. Let
$\pi_{M}^{E}:E\rightarrow M$ be a smooth mapping of Fr\"{o}licher spaces.
Given a Weil Algebra $W=\mathcal{C}^{\infty}(\mathbb{R}^{n})/I$, we will
construct the \textit{Weil prolongation} $\pi_{M}^{E\otimes_{M}W}=\pi_{M}%
^{E}\overrightarrow{\otimes}W:E\otimes_{M}W\rightarrow M$ \textit{of} the
mapping $\pi_{M}^{E}:E\rightarrow M$ \textit{with respect to} $W$. We will
first define $E\otimes_{M}W$ set-theoretically. We define an equivalence
relation $\equiv\operatorname{mod}\ I$ on $\mathcal{C}^{\infty}(\mathbb{R}%
^{n},E)_{M}=\left\{  f\in\mathcal{C}^{\infty}(\mathbb{R}^{n},E)\mid\pi_{M}%
^{E}\circ f\text{ is constant}\right\}  $\ to be
\begin{align*}
f  &  \equiv g\quad\operatorname{mod}\ I\\
&  \text{iff}\\
f(0,...,0)  &  =g(0,...,0)\text{ and}\\
\chi\circ f-\chi\circ g  &  \in I\text{ for every }\chi\in\mathcal{C}^{\infty
}(E,\mathbb{R})
\end{align*}
for any $f,g\in\mathcal{C}^{\infty}(\mathbb{R}^{n},E)_{M}$. The totality of
equivalence classes with respect to the equivaleence relation $\equiv
\operatorname{mod}\ I$ is denoted by $E\otimes_{M}W$, which has the canonical
projection $\pi_{M}^{E\otimes_{M}W}:E\otimes_{M}W\rightarrow M$. This
construction of $E\otimes_{M}W$ can naturally be extended to a functor
$\cdot\overrightarrow{\otimes}W:\mathbf{FS}^{\rightarrow}\mathbf{\rightarrow
Sets}^{\rightarrow}$, where the category $\mathbf{FS}^{\rightarrow}$ is the
category of diagrams in $\mathbf{FS}$\ over the underlying category
\[
\cdot\rightarrow\cdot
\]
and similarly for $\mathbf{Sets}^{\rightarrow}$. In other words, given a
morphism
\[%
\begin{array}
[c]{ccccc}
&  & \varphi &  & \\
& E & \rightarrow & F & \\
\pi_{M}^{E} & \downarrow &  & \downarrow & \pi_{N}^{F}\\
& M & \rightarrow & N & \\
&  & \underline{\varphi} &  &
\end{array}
\]
in the category $\mathbf{FS}^{\rightarrow}$, we have its induced morphism
\[%
\begin{array}
[c]{ccccc}
&  & \varphi\overrightarrow{\otimes}W &  & \\
& E\otimes_{M}W & \rightarrow & F\otimes_{N}W & \\
\pi_{M}^{E\otimes_{M}W} & \downarrow &  & \downarrow & \pi_{N}^{F\otimes_{N}%
W}\\
& M & \rightarrow & N & \\
&  & \underline{\varphi} &  &
\end{array}
\]
in the category $\mathbf{Sets}^{\rightarrow}$. We note that if $M=1$, then
$E\otimes_{M}W$\ is no other than $E\otimes W$, whose construction has already
been discussed in Section 3 of \cite{nishi1}. We endow $E\otimes_{M}W $ with
the initial smooth structure with respect to the mappings
\[
X\otimes_{M}W\overset{\chi\overrightarrow{\otimes}W}{\rightarrow}%
\mathbb{R}\otimes W
\]
where $\chi$ ranges over $\mathcal{C}^{\infty}(X,\mathbb{R})$, and
$\mathbb{R}$\ with the canonical smooth structure is also regarded as an
object $\mathbb{R}\rightarrow1$ in the category $\mathbf{FS}^{\rightarrow}$.
By doing so, we can lift the functor $\cdot\overrightarrow{\otimes
}W:\mathbf{FS}^{\rightarrow}\rightarrow\mathbf{Sets}^{\rightarrow}$ to the
functor $\mathbf{FS}^{\rightarrow}\rightarrow\mathbf{FS}^{\rightarrow}$. As in
Proposition 5 of \cite{nishi1}, we finally have the bifunctor $\overrightarrow
{\otimes}:\mathbf{FS}^{\rightarrow}\times\mathbf{W\rightarrow FS}%
^{\rightarrow}$, which gives rise to a bifunctor $\otimes_{M}:\mathbf{FS}%
/M\times\mathbf{W\rightarrow}$ $\mathbf{FS}/M$ by restriction of the category
$\mathbf{FS}^{\rightarrow}$ to its faithful subcategory $\mathbf{FS}/M$ (the
slice category of $\mathbf{FS}$\ over $M$).

As in Theorem 10 of \cite{nishi1}, we have

\begin{theorem}
\label{t3.1}Given a Weil algebra $W$, the functor $\cdot\otimes_{M}%
W:\mathbf{FS}/M\mathbf{\rightarrow FS}/M$ is product-preserving.
\end{theorem}

From now on through the end of this section, $M$ shall be a Weil exponentiable
and microlinear Fr\"{o}licher space. Now we are going to determine the Weil
prolongation $\pi_{M}^{\left(  M\otimes W_{1}\right)  \otimes_{M}W_{2}%
}:\left(  M\otimes W_{1}\right)  \otimes_{M}W_{2}\rightarrow M$ of the
canonical projection $\pi_{M}^{M\otimes W_{1}}:M\otimes W_{1}\rightarrow M$
with respect to $W_{2}$, where $W_{1}$\ and $W_{2}$\ are arbitrary Weil algebras.

\begin{notation}
Given Weil algebras $W_{1}$\ and $W_{2}$, the equalizer of $\mathcal{W}%
_{d\in\mathcal{D}_{W_{2}}\mapsto(0,d)\in\mathcal{D}_{W_{1}}\times
\mathcal{D}_{W_{2}}}:W_{1}\otimes_{\infty}W_{2}\rightarrow W_{2}$ and
$\mathcal{W}_{d\in\mathcal{D}_{W_{2}}\mapsto(0,0)\in\mathcal{D}_{W_{1}}%
\times\mathcal{D}_{W_{2}}}:W_{1}\otimes_{\infty}W_{2}\rightarrow W_{2}$ in the
category $\mathbf{W}$\ is denoted by $W_{1}\widetilde{\otimes}_{\infty}W_{2}$.
\end{notation}

Now we are going to determine $\mathcal{W}_{D}\widetilde{\otimes}_{\infty
}\mathcal{W}_{D}$ by way of example. The following lemma should be obvious.

\begin{lemma}
\label{t3.2}The diagram
\[
\mathcal{W}_{D(2)}
\begin{array}
[c]{c}%
\mathcal{W}_{\left(  d_{1},d_{2}\right)  \in D\times D\mapsto\left(
d_{1},d_{1}d_{2}\right)  \in D(2)}\\
\rightarrow\\
\,
\end{array}
\mathcal{W}_{D\times D}
\begin{array}
[c]{c}%
\mathcal{W}_{d\in D\mapsto(0,d)\in D\times D}\\
\rightrightarrows\\
\mathcal{W}_{d\in D\mapsto(0,0)\in D\times D}%
\end{array}
\mathcal{W}_{D}%
\]
is an equalizer diagram in $\mathbf{W}$.
\end{lemma}

Therefore, we have

\begin{proposition}
\label{t3.3}$\mathcal{W}_{D}\widetilde{\otimes}_{\infty}\mathcal{W}_{D}$ is no
other than $\mathcal{W}_{D(2)}$.
\end{proposition}

\begin{theorem}
\label{t3.4}Given two Weil algebras $W_{1}$\ and $W_{2}$, the Weil
prolongation $\pi_{M}^{\left(  M\otimes W_{1}\right)  \otimes_{M}W_{2}%
}:\left(  M\otimes W_{1}\right)  \otimes_{M}W_{2}\rightarrow M$ of the
canonical projection $\pi_{M}^{M\otimes W_{1}}:M\otimes W_{1}\rightarrow M$
with respect to $W_{2}$ is no other than $\pi_{M}^{M\otimes\left(
W_{1}\widetilde{\otimes}_{\infty}W_{2}\right)  }:M\otimes\left(
W_{1}\widetilde{\otimes}_{\infty}W_{2}\right)  \rightarrow M$.
\end{theorem}

\begin{remark}
It is for this theorem that we need the notion of $W_{1}\widetilde{\otimes
}_{\infty}W_{2}$.
\end{remark}

\begin{proof}
It is easy to see that $\left(  M\otimes W_{1}\right)  \otimes_{M}W_{2}$\ is
no other than the equalizer of
\[
\left(  M\otimes W_{1}\right)  \otimes W_{2}=M\otimes\left(  W_{1}%
\otimes_{\infty}W_{2}\right)
\begin{array}
[c]{c}%
\mathrm{id}_{M}\otimes\mathcal{W}_{d\in\mathcal{D}_{W_{2}}\mapsto
(0,d)\in\mathcal{D}_{W_{1}}\times\mathcal{D}_{W_{2}}}\\
\rightrightarrows\\
\mathrm{id}_{M}\otimes\mathcal{W}_{d\in\mathcal{D}_{W_{2}}\mapsto
(0,0)\in\mathcal{D}_{W_{1}}\times\mathcal{D}_{W_{2}}}%
\end{array}
M\otimes W_{2}%
\]
However, since $M$\ is microlinear, the equalizer diagram
\[
W_{1}\widetilde{\otimes}_{\infty}W_{2}\rightarrow W_{1}\otimes_{\infty}W_{2}
\begin{array}
[c]{c}%
\mathcal{W}_{d\in\mathcal{D}_{W_{2}}\mapsto(0,d)\in\mathcal{D}_{W_{1}}%
\times\mathcal{D}_{W_{2}}}\\
\rightrightarrows\\
\mathcal{W}_{d\in\mathcal{D}_{W_{2}}\mapsto(0,0)\in\mathcal{D}_{W_{1}}%
\times\mathcal{D}_{W_{2}}}%
\end{array}
W_{2}%
\]
in the category $\mathbf{W}$\ naturally gives rise to the equalizer diagram
\[
M\otimes\left(  W_{1}\widetilde{\otimes}_{\infty}W_{2}\right)  \rightarrow
M\otimes\left(  W_{1}\otimes_{\infty}W_{2}\right)
\begin{array}
[c]{c}%
\mathrm{id}_{M}\otimes\mathcal{W}_{d\in\mathcal{D}_{W_{2}}\mapsto
(0,d)\in\mathcal{D}_{W_{1}}\times\mathcal{D}_{W_{2}}}\\
\rightrightarrows\\
\mathrm{id}_{M}\otimes\mathcal{W}_{d\in\mathcal{D}_{W_{2}}\mapsto
(0,0)\in\mathcal{D}_{W_{1}}\times\mathcal{D}_{W_{2}}}%
\end{array}
M\otimes W_{2}%
\]
in the category $\mathbf{FS}$. Therefore, we are now sure that the Weil
prolongation $\pi_{M}^{\left(  M\otimes W_{1}\right)  \otimes_{M}W_{2}%
}:\left(  M\otimes W_{1}\right)  \otimes_{M}W_{2}\rightarrow M$ of the
canonical projection $\pi_{M}^{M\otimes W_{1}}:M\otimes W_{1}\rightarrow M$
with respect to $W_{2}$ is no other than $\pi_{M}^{M\otimes\left(
W_{1}\widetilde{\otimes}_{\infty}W_{2}\right)  }:M\otimes\left(
W_{1}\widetilde{\otimes}_{\infty}W_{2}\right)  \rightarrow M$.
\end{proof}

\section{Relatived Weil Exponentiability}

As in Lemma \ref{t2.3.1}, we have

\begin{lemma}
\label{t4.1}For any smooth mappings $\pi_{M}^{E}:E\rightarrow M$\ and $\pi
_{M}^{F}:F\rightarrow M$\ and for any Weil algebra $W$, we have
\[
\left[  E,F\right]  _{M}\otimes_{M}W=\left[  E,F\otimes_{M}W\right]  _{M}%
\]

\end{lemma}

Now we relativize the notion of Weil exponentiability.

\begin{definition}
A smooth mapping $\pi_{M}^{E}:E\rightarrow M$\ of Fr\"{o}licher spaces over
$M$\ is called \textit{Weil-exponentiable with respect to }$M$\textit{\ }or,
more briefly,\textit{\ }$M$-\textit{Weil-exponentiable }if
\[
E\otimes_{M}(W_{1}\otimes_{\infty}W_{2})=\left(  E\otimes_{M}W_{1}\right)
\otimes_{M}W_{2}%
\]
holds naturally for any Weil algebras $W_{1}$ and $W_{2}$.
\end{definition}

As in Proposition \ref{t2.3.2}, we have

\begin{proposition}
\label{t4.2}If a smooth mapping $\pi_{M}^{E}:E\rightarrow M$\ of Fr\"{o}licher
spaces is $M$-Weil-exponentiable, then so is $\pi_{M}^{\left[  F,E\right]
_{M}}:\left[  F,E\right]  _{M}\rightarrow M$\ for any smooth mapping $\pi
_{M}^{F}:F\rightarrow M$\ of Fr\"{o}licher spaces.
\end{proposition}

As in Proposition \ref{t2.3.3}, we have

\begin{proposition}
\label{t4.3}If a smooth mapping $\pi_{M}^{E}:E\rightarrow M$\ of Fr\"{o}licher
spaces is $M$-Weil-exponentiable, then so is its $M$-Weil prolongation
$\pi_{M}^{E\otimes_{M}W}:E\otimes_{M}W\rightarrow M$ with respect to any Weil
algebra $W$.
\end{proposition}

As in Proposition \ref{t2.3.4}, we have

\begin{proposition}
\label{t4.4}If $\pi_{M}^{E}:E\rightarrow M$ and $\pi_{M}^{F}:F\rightarrow M $
are $M$-Weil-exponentiable smooth mappings of Fr\"{o}licher spaces, then so is
$\pi_{M}^{E\times_{M}F}:E\times_{M}F\rightarrow M$.
\end{proposition}

\begin{theorem}
\label{t4.5}The full subcategory $(\mathbf{FS}/M)_{\mathrm{WE}}$\ of all
$M$-Weil-exponentiable smooth mappings of Fr\"{o}licher spaces of the category
$\mathbf{FS}/M$ is Cartesian closed.
\end{theorem}

We conclude this section by noting that the canonical projection $\pi
_{M}^{M\otimes W}:M\otimes W\rightarrow M$ is $M$-Weil exponentiable for any
Weil algebra $W$, where $M$ is a Weil exponentiable and microlinear
Fr\"{o}licher space. To this end, we need

\begin{proposition}
\label{t4.6}Given Weil algebras $W_{1}$, $W_{2}$\ and $W_{3}$, we have
\[
\left(  W_{1}\widetilde{\otimes}_{\infty}W_{2}\right)  \widetilde{\otimes
}_{\infty}W_{3}=W_{1}\widetilde{\otimes}_{\infty}\left(  W_{2}\otimes_{\infty
}W_{3}\right)
\]

\end{proposition}

\begin{proof}
Since the functor $\cdot\otimes_{\infty}W:\mathbf{W}\rightarrow\mathbf{W} $
preserves limits, particularly equalizers, $\left(  W_{1}\widetilde{\otimes
}_{\infty}W_{2}\right)  \otimes_{\infty}W_{3}$ is no other than the equalizer
of
\[
W_{1}\otimes_{\infty}W_{2}\otimes_{\infty}W_{3}
\begin{array}
[c]{c}%
\mathcal{W}_{\left(  d_{2},d_{3}\right)  \in\mathcal{D}_{W_{2}}\times
\mathcal{D}_{W_{3}}\mapsto(0,d_{2},d_{3})\in\mathcal{D}_{W_{1}}\times
\mathcal{D}_{W_{2}}\times\mathcal{D}_{W_{3}}}\\
\rightrightarrows\\
\mathcal{W}_{\left(  d_{2},d_{3}\right)  \in\mathcal{D}_{W_{2}}\times
\mathcal{D}_{W_{3}}\mapsto(0,0,d_{3})\in\mathcal{D}_{W_{1}}\times
\mathcal{D}_{W_{2}}\times\mathcal{D}_{W_{3}}}%
\end{array}
W_{2}\otimes_{\infty}W_{3}%
\]
Since the canonical injection $W_{3}\rightarrow W_{2}\otimes_{\infty}W_{3}$ is
represented by
\[
W_{3}
\begin{array}
[c]{c}%
\mathcal{W}_{\left(  d_{2},d_{3}\right)  \in\mathcal{D}_{W_{2}}\times
\mathcal{D}_{W_{3}}\mapsto d_{3}\in\mathcal{D}_{W_{3}}}\\
\rightarrow\\
\,
\end{array}
W_{2}\otimes_{\infty}W_{3}%
\]
the equalizer of
\[
\left(  W_{1}\widetilde{\otimes}_{\infty}W_{2}\right)  \otimes_{\infty}W_{3}
\begin{array}
[c]{c}%
\mathcal{W}_{d_{3}\in\mathcal{D}_{W_{3}}\mapsto(0,d_{3})\in\mathcal{D}%
_{W_{1}\widetilde{\otimes}_{\infty}W_{2}}\times\mathcal{D}_{W_{3}}}\\
\rightrightarrows\\
\mathcal{W}_{d_{3}\in\mathcal{D}_{W_{3}}\mapsto(0,0)\in\mathcal{D}%
_{W_{1}\widetilde{\otimes}_{\infty}W_{2}}\times\mathcal{D}_{W_{3}}}%
\end{array}
W_{3}%
\]
is no other than the intersection of $\left(  W_{1}\widetilde{\otimes}%
_{\infty}W_{2}\right)  \otimes_{\infty}W_{3}$ with the equalizer of
\[
W_{1}\otimes_{\infty}W_{2}\otimes_{\infty}W_{3}
\begin{array}
[c]{c}%
\mathcal{W}_{\left(  d_{2},d_{3}\right)  \in\mathcal{D}_{W_{2}}\times
\mathcal{D}_{W_{3}}\mapsto(0,0,d_{3})\in\mathcal{D}_{W_{1}}\times
\mathcal{D}_{W_{2}}\times\mathcal{D}_{W_{3}}}\\
\rightrightarrows\\
\mathcal{W}_{\left(  d_{2},d_{3}\right)  \in\mathcal{D}_{W_{2}}\times
\mathcal{D}_{W_{3}}\mapsto(0,0,0)\in\mathcal{D}_{W_{1}}\times\mathcal{D}%
_{W_{2}}\times\mathcal{D}_{W_{3}}}%
\end{array}
W_{2}\otimes_{\infty}W_{3}%
\]
Therefore, we are sure that $\left(  W_{1}\widetilde{\otimes}_{\infty}%
W_{2}\right)  \widetilde{\otimes}_{\infty}W_{3}$ is the equalizer of
\[
W_{1}\otimes_{\infty}W_{2}\otimes_{\infty}W_{3}
\begin{array}
[c]{c}%
\mathcal{W}_{\left(  d_{2},d_{3}\right)  \in\mathcal{D}_{W_{2}}\times
\mathcal{D}_{W_{3}}\mapsto(0,d_{2},d_{3})\in\mathcal{D}_{W_{1}}\times
\mathcal{D}_{W_{2}}\times\mathcal{D}_{W_{3}}}\\
\rightarrow\\
\mathcal{W}_{\left(  d_{2},d_{3}\right)  \in\mathcal{D}_{W_{2}}\times
\mathcal{D}_{W_{3}}\mapsto(0,0,d_{3})\in\mathcal{D}_{W_{1}}\times
\mathcal{D}_{W_{2}}\times\mathcal{D}_{W_{3}}}\\
\rightarrow\\
\mathcal{W}_{\left(  d_{2},d_{3}\right)  \in\mathcal{D}_{W_{2}}\times
\mathcal{D}_{W_{3}}\mapsto(0,0,0)\in\mathcal{D}_{W_{1}}\times\mathcal{D}%
_{W_{2}}\times\mathcal{D}_{W_{3}}}\\
\rightarrow\\
\,
\end{array}
W_{2}\otimes_{\infty}W_{3}%
\]
which is no other than the equalizer of
\[
W_{1}\otimes_{\infty}W_{2}\otimes_{\infty}W_{3}
\begin{array}
[c]{c}%
\mathcal{W}_{\left(  d_{2},d_{3}\right)  \in\mathcal{D}_{W_{2}}\times
\mathcal{D}_{W_{3}}\mapsto(0,d_{2},d_{3})\in\mathcal{D}_{W_{1}}\times
\mathcal{D}_{W_{2}}\times\mathcal{D}_{W_{3}}}\\
\rightrightarrows\\
\mathcal{W}_{\left(  d_{2},d_{3}\right)  \in\mathcal{D}_{W_{2}}\times
\mathcal{D}_{W_{3}}\mapsto(0,0,0)\in\mathcal{D}_{W_{1}}\times\mathcal{D}%
_{W_{2}}\times\mathcal{D}_{W_{3}}}%
\end{array}
W_{2}\otimes_{\infty}W_{3}%
\]
This implies the desired result, completing the proof.
\end{proof}

\begin{theorem}
\label{t4.7}For any Weil algebra $W$, the canonical projection $\pi
_{M}^{M\otimes W}:M\otimes W\rightarrow M$ is $M$-Weil exponentiable.
\end{theorem}

\begin{proof}
We have
\begin{align*}
&  \left(  \left(  M\otimes W\right)  \otimes_{M}W_{1}\right)  \otimes
_{M}W_{2}\\
&  =\left(  M\otimes\left(  W\widetilde{\otimes}_{\infty}W_{1}\right)
\right)  \otimes_{M}W_{2}\\
&  \text{[by Theorem \ref{t3.4}]}\\
&  =M\otimes\left(  \left(  W\widetilde{\otimes}_{\infty}W_{1}\right)
\widetilde{\otimes}_{\infty}W_{2}\right) \\
&  \text{[by Theorem \ref{t3.4}]}\\
&  =M\otimes\left(  W\widetilde{\otimes}_{\infty}\left(  W_{1}\otimes_{\infty
}W_{2}\right)  \right) \\
&  \text{[by Proposition \ref{t4.6}]}\\
&  =\left(  M\otimes W\right)  \otimes_{M}\left(  W_{1}\otimes_{\infty}%
W_{2}\right) \\
&  \text{[by Theorem \ref{t3.4}]}%
\end{align*}

\end{proof}

\section{Relativized Microlinearity}

\begin{definition}
A smooth mapping $\pi_{M}^{E}:E\rightarrow M$\ of Fr\"{o}licher spaces is
called microlinear with respect to $M$\ or, more briefly,\ $M$%
-microlinear\ providing that any finite limit diagram $\mathbb{D}$ in $W$
yields a limit diagram $E\otimes_{M}\mathbb{D}$ in the category $\mathbf{FS}$,
where $E\otimes_{M}\mathbb{D}$ is obtained from $\mathbb{D}$ by putting
$E\otimes_{M}$ to the left of every object in $\mathbb{D}$ and putting
$\mathrm{id}_{E}\otimes$ to the left of every morphism in $\mathbb{D}$.
\end{definition}

As in Proposition 14 of \cite{nishi2}, we have

\begin{proposition}
\label{t5.2}If $\pi_{M}^{E}:E\rightarrow M$ is a $M$-Weil exponentiable and
$M$-microlinear smooth mapping of Fr\"{o}licher spaces, then so is $\pi
_{M}^{\left[  F,E\right]  _{M}}:\left[  F,E\right]  _{M}\rightarrow M$ for any
smooth mapping $\pi_{M}^{F}:F\rightarrow M$\ of Fr\"{o}licher spaces.
\end{proposition}

As in Proposition 12 of \cite{nishi2}, we have

\begin{proposition}
\label{t5.3}If $\pi_{M}^{E}:E\rightarrow M$ is a $M$-Weil exponentiable and
$M$-microlinear smooth mapping of Fr\"{o}licher spaces, then so is $\pi
_{M}^{E\otimes_{M}W}:E\otimes_{M}W\rightarrow M$ for any Weil algebra $W$.
\end{proposition}

As in Proposition 13 of \cite{nishi2}, we have

\begin{proposition}
\label{t5.4}If $\pi_{M}^{E}:E\rightarrow M$ and $\pi_{M}^{F}:F\rightarrow M $
are $M$-microlinear smooth mappings of Fr\"{o}licher spaces, then so is
$\pi_{M}^{E\times_{M}F}:E\times_{M}F\rightarrow M$.
\end{proposition}

\begin{theorem}
\label{t5.5}The full subcategory $(\mathbf{FS}/M)_{\mathrm{WE,ML}}$\ of all
$M$-Weil exponentiable and $M$-microlinear smooth mappings of Fr\"{o}licher
spaces ($(\mathbf{FS}/M)_{\mathrm{ML}}$\ of all $M$-microlinear smooth
mappings of Fr\"{o}licher spaces, resp.) of the category $\mathbf{FS}/M$ is
cartesian closed.
\end{theorem}

\begin{proof}
This follows from Theorem \ref{t2.1.3} and Propositions \ref{t4.2},
\ref{t4.4}, \ref{t5.2} and \ref{t5.4}.
\end{proof}

Now it is appropriate to introduce the following definition.

\begin{definition}
A smooth mapping $\pi_{M}^{E}:E\rightarrow M$\ of Fr\"{o}licher spaces is
called a fiber bundle over $M$\ provided that it is $M$-Weil exponentiable and
$M$-microlinear.
\end{definition}

\begin{notation}
The category $(\mathbf{FS}/M)_{\mathrm{WE,ML}}$ is also denoted by
$\mathbf{Fib}_{M}$.
\end{notation}

Now we are going to show that

\begin{theorem}
\label{t5.6}Let $M$ be a Weil exponentiable and microlinear Fr\"{o}licher
space. The canonical projection $\pi_{M}^{M\otimes W}:M\otimes W\rightarrow M
$ is $M$-microlinear for any Weil algebra $W$.
\end{theorem}

To this end, we need

\begin{lemma}
\label{t5.7}Given a Weil algebra $W$\ and a finite diagram $\mathbb{D}$\ of
Weil algebras, we have
\[
\mathrm{Lim\,}\left(  W\widetilde{\otimes}_{\infty}\mathbb{D}\right)
=W\widetilde{\otimes}_{\infty}\mathrm{Lim\,}\mathbb{D}%
\]

\end{lemma}

\begin{proof}
This follows simply from the well-known two facts that the functor
$\cdot\otimes_{\infty}W:\mathbf{W}\rightarrow\mathbf{W}$ is left-exact and
that double limits commute.
\end{proof}

\begin{proof}
\textit{(of Theorem \ref{t5.5})} For any finite diagram $\mathbb{D}$\ of Weil
algebras, we have
\begin{align*}
&  \mathrm{Lim\,}\left(  \left(  M\otimes W\right)  \otimes_{M}\mathbb{D}%
\right) \\
&  =\mathrm{Lim\,}\left(  M\otimes\left(  W\widetilde{\otimes}_{\infty
}\mathbb{D}\right)  \right) \\
&  \text{[by Theorem \ref{t3.4}]}\\
&  =M\otimes\mathrm{Lim\,}\left(  W\widetilde{\otimes}_{\infty}\mathbb{D}%
\right) \\
&  \text{[since }M\text{\ is microlinear]}\\
&  =M\otimes\left(  W\widetilde{\otimes}_{\infty}\mathrm{Lim\,}\mathbb{D}%
\right) \\
&  \text{[by Lemma \ref{t5.6}]}\\
&  =\left(  M\otimes W\right)  \otimes_{M}\mathrm{Lim\,}\mathbb{D}%
\end{align*}
It may be meaningful to state here that
\end{proof}

\begin{theorem}
Let $M$ be a Weil exponentiable and microlinear Fr\"{o}licher space. The
canonical projection $\pi_{M}^{M\otimes W}:M\otimes W\rightarrow M$ is a fiber
bundle for any Weil algebra $W$.
\end{theorem}

\begin{proof}
This follows from Theorems \ref{t4.7} and \ref{t5.6}.
\end{proof}

\section{Vector Bundles}

Let us begin this section with two definitions.

\begin{definition}
A smooth mapping $\pi_{M}^{E}:E\rightarrow M$\ of Fr\"{o}licher spaces is
called a prevector bundle over $M$\ providing that it is endowed with a linear
structure in the category $\mathbf{FS}/M$.
\end{definition}

\begin{definition}
A prevector bundle $\pi_{M}^{E}:E\rightarrow M$\ is called a vector bundle
over $M$\ provided that it is $M$-Euclidean in the sense that
\[
E\otimes_{M}\mathcal{W}_{D}=E\times_{M}E
\]
holds naturally.
\end{definition}

We restate Theorem \ref{t2.1.4}\ in the following way.

\begin{theorem}
\label{t6.0}The totality $\mathbf{preVect}_{M}$ of prevector bundles over
$M$\ and their $M$-linear smooth mappings over $M$\ forms a cartesian closed category.
\end{theorem}

As in Lemma \ref{t2.3.1}, we have

\begin{lemma}
\label{t6.1}For any objects $\pi_{M}^{E}:E\rightarrow M$\ and $\pi_{M}%
^{F}:F\rightarrow M$\ in the category $\mathbf{preVect}_{M}$\ and for any Weil
algebra $W$, we have
\[
\left[  E,F\right]  _{M}^{\mathrm{Lin}}\otimes_{M}W=\left[  E,F\otimes
_{M}W\right]  _{M}^{\mathrm{Lin}}%
\]

\end{lemma}

As in Proposition 2 and Proposition 10 of \cite{nishi2}, we have

\begin{proposition}
\label{t6.2}A prevector bundle $\pi_{M}^{E}:E\rightarrow M$\ is $M$-Weil
exponentiable and $M$-microlinear, so that it is a fiber bundle over $M$.
\end{proposition}

Now we are going to show that the category $\mathbf{Vect}_{M}$ of vector
bundles over $M$\ and their $M$-linear smooth mappings over $M$ is cartesian
closed. To this end, we have to establish the following two propositions.

\begin{proposition}
\label{t6.3}If both $\pi_{M}^{E}:E\rightarrow M$\ and $\pi_{M}^{F}%
:F\rightarrow M$\ are vector bundles, then so is their product $\pi
_{M}^{E\times_{M}F}:E\times_{M}F\rightarrow M$.
\end{proposition}

\begin{proof}
We have to show that $\pi_{M}^{E\times_{M}F}:E\times_{M}F\rightarrow M$\ is
$M$-Euclidean, for which we have
\begin{align*}
&  \left(  E\times_{M}F\right)  \otimes_{M}\mathcal{W}_{D}\\
&  =\left(  E\otimes_{M}\mathcal{W}_{D}\right)  \times_{M}\left(  F\otimes
_{M}\mathcal{W}_{D}\right) \\
&  \text{\lbrack by Theorem \ref{t3.1}]}\\
&  =\left(  E\times_{M}E\right)  \times_{M}\left(  F\times_{M}F\right) \\
&  =\left(  E\times_{M}F\right)  \times_{M}\left(  E\times_{M}F\right)
\end{align*}
This completes the proof.
\end{proof}

\begin{proposition}
\label{t6.4}If both $\pi_{M}^{E}:E\rightarrow M$\ and $\pi_{M}^{F}%
:F\rightarrow M$\ are vector bundles, then so is their exponential $\pi
_{M}^{\left[  E,F\right]  _{M}^{\mathrm{Lin}}}:\left[  E,F\right]
_{M}^{\mathrm{Lin}}\rightarrow M$.
\end{proposition}

\begin{proof}
We have to show that $\pi_{M}^{\left[  E,F\right]  _{M}^{\mathrm{Lin}}%
}:\left[  E,F\right]  _{M}^{\mathrm{Lin}}\rightarrow M$\ is $M$-Euclidean, for
which we have
\begin{align*}
&  \left[  E,F\right]  _{M}^{\mathrm{Lin}}\otimes_{M}\mathcal{W}_{D}\\
&  =\left[  E,F\otimes_{M}\mathcal{W}_{D}\right]  _{M}^{\mathrm{Lin}}\\
&  \text{\lbrack by Lemma \ref{t6.1}]}\\
&  =\left[  E,F\times_{M}F\right]  _{M}^{\mathrm{Lin}}\\
&  =\left[  E,F\right]  _{M}^{\mathrm{Lin}}\times_{M}\left[  E,F\right]
_{M}^{\mathrm{Lin}}\\
&  \text{[since the functor }\left[  E,\cdot\right]  _{M}^{\mathrm{Lin}%
}:\mathbf{preVect}_{M}\rightarrow\mathbf{preVect}_{M}\text{\ is
product-preserving]}%
\end{align*}

\end{proof}

\begin{theorem}
\label{t6.5}The category $\mathbf{Vect}_{M}$\ is cartesian closed.
\end{theorem}

\begin{proof}
This follows from Theorem \ref{t6.0}\ and Propositions \ref{t6.3}\ and
\ref{t6.4}.
\end{proof}

To conclude this section, we would like to establish that

\begin{theorem}
\label{t6.6}Let $M$ be a Weil exponentiable and microlinear Fr\"{o}licher
space. The canonical projection $\pi_{M}^{M\otimes\mathcal{W}_{D}}%
:M\otimes\mathcal{W}_{D}\rightarrow M$ is naturally a vector bundle.
\end{theorem}

\begin{proof}
We have already explained in detail in Theorem 3 of \cite{nishi3} how the
object $\pi_{M}^{M\otimes\mathcal{W}_{D}}:M\otimes\mathcal{W}_{D}\rightarrow
M$ in the category $\mathbf{FS}/M$\ is naturally endowed with a linear
structure. We have to show that it is $M$-Euclidean, for which we have
\begin{align*}
&  \left(  M\otimes\mathcal{W}_{D}\right)  \otimes_{M}\mathcal{W}_{D}\\
&  =M\otimes\left(  \mathcal{W}_{D}\widetilde{\otimes}_{\infty}\mathcal{W}%
_{D}\right) \\
&  \text{\lbrack by Theorem \ref{t3.4}]}\\
&  =M\otimes\mathcal{W}_{D(2)}\\
&  \text{\lbrack by Proposition \ref{t3.3}]}\\
&  =\left(  M\otimes\mathcal{W}_{D}\right)  \times_{M}\left(  M\otimes
\mathcal{W}_{D}\right)
\end{align*}

\end{proof}

\end{document}